\pdfoutput=1 
\documentclass[a4paper]{article}
\usepackage{amsthm,amssymb,amsmath,enumerate,graphicx}
\usepackage{tikz}
\usepackage{xr}

\newtheorem{THM}{Theorem}[section]
\newtheorem{LEM}[THM]{Lemma}

\theoremstyle{definition}
\newtheorem{EX}[THM]{Example}

\DeclareMathOperator{\Small}{Small}

\def\shift(#1)(#2){\!\!\downarrow\!{}^{#1}_{\raise .1ex\vbox to 0pt{\vss\hbox{$\scriptstyle #2$}}}\,}
\def\ucl(#1){\lfloor #1 \rfloor}
\def\dcl(#1){\lceil #1 \rceil}
\def\specrel#1#2{\mathrel{\mathop{\kern0pt #1}\limits_{#2}}}

\newcommand\B{\mathcal B}

\def\S{\mathcal S}

\def\lowfwd #1#2#3{{\mathop{\kern0pt #1}\limits^{\kern#2pt\raise.#3ex
\vbox to 0pt{\hbox{$\scriptscriptstyle\rightarrow$}\vss}}}}
\def\lowbkwd #1#2#3{{\mathop{\kern0pt #1}\limits^{\kern#2pt\raise.#3ex
\vbox to 0pt{\hbox{$\scriptscriptstyle\leftarrow$}\vss}}}}
\def\fwd #1#2{{\lowfwd{#1}{#2}{15}}}
\def\ve{\kern-1.5pt\lowfwd e{1.5}2\kern-1pt}
\def\vedash{{\mathop{\kern0pt e\lower.5pt\hbox{${}
     \scriptstyle'$}}\limits^{\kern0pt\raise.02ex
     \vbox to 0pt{\hbox{$\scriptscriptstyle\rightarrow$}\vss}}}}
\def\ev{\kern-1pt\lowbkwd e{0.5}2\kern-1pt}
\def\vf{\kern-2pt\lowfwd f{2.5}2\kern-1pt}
\def\vfdash{{\mathop{\kern0pt f\raise 1pt\hbox{${}
     \scriptstyle'$}}\limits^{\kern2pt\raise.02ex
     \vbox to 0pt{\hbox{$\scriptscriptstyle\rightarrow$}\vss}}}}

\def\vr{\lowfwd r{1.5}2}
\def\rv{\lowbkwd r02}

\def\vrdash{{\mathop{\kern0pt r\lower.5pt\hbox{${}
     \scriptstyle'$}}\limits^{\kern0pt\raise.02ex
     \vbox to 0pt{\hbox{$\scriptscriptstyle\rightarrow$}\vss}}}}
\def\rvdash{{\mathop{\kern0pt r\lower.5pt\hbox{${}
     \scriptstyle'$}}\limits^{\kern0pt\raise.02ex
     \vbox to 0pt{\hbox{$\scriptscriptstyle\leftarrow$}\vss}}}}
\def\vrdashp{{\mathop{\kern0pt r_p\kern-4pt\lower.5pt\hbox{${}
     \scriptstyle'$}}\limits^{\kern0pt\raise.02ex
     \vbox to 0pt{\hbox{$\scriptscriptstyle\rightarrow$}\vss}}}\,}
\def\rvdashp{{\mathop{\kern0pt r_p\kern-4pt\lower.5pt\hbox{${}
     \scriptstyle'$}}\limits^{\kern0pt\raise.02ex
     \vbox to 0pt{\hbox{$\scriptscriptstyle\leftarrow$}\vss}}}\,}
\def\vrddash{{\mathop{\kern0pt r\lower.5pt\hbox{${}
     \scriptstyle''$}}\limits^{\kern0pt\raise.02ex
     \vbox to 0pt{\hbox{$\scriptscriptstyle\rightarrow$}\vss}}}}

\def\vs{\lowfwd s{1.5}1}
\def\sv{{{\lowbkwd s{1.5}1}\hskip-1pt}}
\def\vsdash{{\mathop{\kern0pt s\lower.5pt\hbox{${}
     \scriptstyle'$}}\limits^{\kern0pt\raise.02ex
     \vbox to 0pt{\hbox{$\scriptscriptstyle\rightarrow$}\vss}}}}
\def\svdash{{\mathop{\kern0pt s\lower.5pt\hbox{${}
     \scriptstyle'$}}\limits^{\kern0pt\raise.02ex
     \vbox to 0pt{\hbox{$\scriptscriptstyle\leftarrow$}\vss}}}}
\def\vsddash{{\mathop{\kern0pt s\lower.5pt\hbox{${}
     \scriptstyle''$}}\limits^{\kern0pt\raise.02ex
     \vbox to 0pt{\hbox{$\scriptscriptstyle\rightarrow$}\vss}}}}
\def\svddash{{\mathop{\kern0pt s\lower.5pt\hbox{${}
     \scriptstyle''$}}\limits^{\kern0pt\raise.02ex
     \vbox to 0pt{\hbox{$\scriptscriptstyle\leftarrow$}\vss}}}}
\def\vsdashp{{\mathop{\kern0pt s_p\kern-4pt\lower.5pt\hbox{${}
     \scriptstyle'$}}\limits^{\kern0pt\raise.02ex
     \vbox to 0pt{\hbox{$\scriptscriptstyle\rightarrow$}\vss}}}\,}
\def\svdashp{{\mathop{\kern0pt s_p\kern-4pt\lower.5pt\hbox{${}
     \scriptstyle'$}}\limits^{\kern0pt\raise.02ex
     \vbox to 0pt{\hbox{$\scriptscriptstyle\leftarrow$}\vss}}}\,}

\def\vsidash{{\mathop{\kern0pt s_i\kern-3.5pt\lower.3pt\hbox{${}
     \scriptstyle'$}}\limits^{\kern0pt\raise.02ex
     \vbox to 0pt{\hbox{$\scriptscriptstyle\rightarrow$}\vss}}}}

\def\vsqdash{{\mathop{\kern0pt s_q\kern-3.5pt\lower.3pt\hbox{${}
     \scriptstyle'$}}\limits^{\kern0pt\raise.02ex
     \vbox to 0pt{\hbox{$\scriptscriptstyle\rightarrow$}\vss}}}}

\def\vS{{\hskip-1pt{\fwd S3}\hskip-1pt}} 

\def\vSstar{{\mathop{\kern0pt S\lower-1pt\hbox{$^*$}}\limits^{\kern2pt
     \vbox to 0pt{\hbox{$\scriptscriptstyle\rightarrow$}\vss}}}}
\def\vSdash{{\mathop{\kern0pt S\lower-1pt\hbox{${}
     \scriptstyle'$}}\limits^{\kern2pt\raise.1ex
     \vbox to 0pt{\hbox{$\scriptscriptstyle\rightarrow$}\vss}}}}
\def\vt{\lowfwd t{1.5}1}
\def\tv{\lowbkwd t{1.5}1}

\def\vU{{\vec U}} 

\def\vO{\mathcal{O}}

\def\sub{\subseteq}

\def\sm{\smallsetminus}

\newcommand\COMMENT[1]{}

\def\?#1{\vadjust{\vbox to 0pt{\vss\vskip-8pt\leftline{%
     \llap{\hbox{\vbox{\pretolerance=-1
     \doublehyphendemerits=0\finalhyphendemerits=0
     \hsize20truemm\tolerance=10000\small
     \lineskip=0pt\lineskiplimit=0pt
     \rightskip=0pt plus16truemm\baselineskip8pt\noindent
     \hskip0pt        
     #1\endgraf}\hskip2truemm}}}\vss}}}

\newenvironment{txteq*}
  {
    \begin{equation*}
    \begin{minipage}[c]{0.85\textwidth} 
    \em                                
  }
  {\end{minipage}\end{equation*}\ignorespacesafterend}

\renewcommand\S{\mathcal S}

\newcommand{\braces}[1]{\left(#1\right)}
\newcommand{\menge}[1]{\left\{#1\right\}}

\newcommand{\class}[1]{\left[#1\right]}

\newcommand{\tn}[1]{\textnormal{#1}}

\def\U{\mathcal{U}}

\def\lowfwd #1#2#3{{\mathop{\kern0pt #1}\limits^{\kern#2pt\raise.#3ex
\vbox to 0pt{\hbox{$\scriptscriptstyle\rightarrow$}\vss}}}}
\def\lowbkwd #1#2#3{{\mathop{\kern0pt #1}\limits^{\kern#2pt\raise.#3ex
\vbox to 0pt{\hbox{$\scriptscriptstyle\leftarrow$}\vss}}}}
\def\fwd #1#2{{\lowfwd{#1}{#2}{15}}}
\def\ve{\kern-1pt\lowfwd e{1.5}2\kern-1pt}
\def\ev{\kern-1pt\lowbkwd e{1.5}2\kern-1pt}

\def\vr{\lowfwd r{1.5}2}
\def\rv{\lowbkwd r02}

\def\vrdash{{\mathop{\kern0pt r\lower.5pt\hbox{${}
     \scriptstyle'$}}\limits^{\kern0pt\raise.02ex
     \vbox to 0pt{\hbox{$\scriptscriptstyle\rightarrow$}\vss}}}}
\def\rvdash{{\mathop{\kern0pt r\lower.5pt\hbox{${}
     \scriptstyle'$}}\limits^{\kern0pt\raise.02ex
     \vbox to 0pt{\hbox{$\scriptscriptstyle\leftarrow$}\vss}}}}

\def\vs{\lowfwd s{1.5}1}
\def\sv{\lowbkwd s{1.5}1}

\def\vsidash{{\mathop{\kern0pt s_i\kern-3.5pt\lower.3pt\hbox{${}
     \scriptstyle'$}}\limits^{\kern0pt\raise.02ex
     \vbox to 0pt{\hbox{$\scriptscriptstyle\rightarrow$}\vss}}}}
\def\vS{{\hskip-1pt{\fwd S3}\hskip-1pt}} 
\def\vSr{{\vec S}_{\raise.1ex\vbox to 0pt{\vss\hbox{$\scriptstyle\ge\vr$}}}}
\def\vSdash{{\mathop{\kern0pt S\lower-1pt\hbox{${}
     \scriptstyle'$}}\limits^{\kern2pt\raise.1ex
     \vbox to 0pt{\hbox{$\scriptscriptstyle\rightarrow$}\vss}}}}
 \def\vUdash{{\mathop{\kern0pt U\lower-1pt\hbox{${}
 				\scriptstyle'$}}\limits^{\kern2pt\raise.1ex
 			\vbox to 0pt{\hbox{$\scriptscriptstyle\rightarrow$}\vss}}}}
\def\vsdash{{\mathop{\kern0pt s\lower.5pt\hbox{${}
     \scriptstyle'$}}\limits^{\kern0pt\raise.02ex
     \vbox to 0pt{\hbox{$\scriptscriptstyle\rightarrow$}\vss}}}}
     
\def\svdash{{\mathop{\kern0pt s\lower.5pt\hbox{${}
     \scriptstyle'$}}\limits^{\kern0pt\raise.02ex
     \vbox to 0pt{\hbox{$\scriptscriptstyle\leftarrow$}\vss}}}}
     
\def\vtdash{{\mathop{\kern0pt t\lower0pt\hbox{${}
     \scriptstyle'$}}\limits^{\kern0pt\raise.1ex
     \vbox to 0pt{\hbox{$\scriptscriptstyle\rightarrow$}\vss}}}}
     
\def\tvdash{{\mathop{\kern0pt t\lower0pt\hbox{${}
     \scriptstyle'$}}\limits^{\kern0pt\raise.1ex
     \vbox to 0pt{\hbox{$\scriptscriptstyle\leftarrow$}\vss}}}}
     
\def\vddash{{\mathop{\kern0pt d\raise1pt\hbox{${}
     \scriptstyle'$}}\limits^{\kern0pt\raise.02ex
     \vbox to 0pt{\hbox{$\scriptscriptstyle\rightarrow$}\vss}}}}
     
\def\dvdash{{\mathop{\kern0pt d\raise1pt\hbox{${}
     \scriptstyle'$}}\limits^{\kern0pt\raise.02ex
     \vbox to 0pt{\hbox{$\scriptscriptstyle\leftarrow$}\vss}}}}
     
\def\vtstar{{\mathop{\kern0pt t\raise2.5pt\hbox{${}
     \scriptstyle*$}}\limits^{\kern0pt\raise.1ex
     \vbox to 0pt{\hbox{$\scriptscriptstyle\rightarrow$}\vss}}}}
     
\def\tvstar{{\mathop{\kern0pt t\raise2.5pt\hbox{${}
     \scriptstyle*$}}\limits^{\kern0pt\raise.1ex
     \vbox to 0pt{\hbox{$\scriptscriptstyle\leftarrow$}\vss}}}}

\def\vtstarD{{\mathop{\kern0pt t\kern.5pt\raise3pt\hbox{${}
     \scriptstyle*$}{\kern-5.5pt\lower3pt\hbox{$
     \scriptstyle D$}}}\limits^{\kern0pt\raise.1ex
     \vbox to 0pt{\hbox{$\scriptscriptstyle\rightarrow$}\vss}}}}

\def\tvstarD{{\mathop{\kern0pt t\kern.5pt\raise3pt\hbox{${}
     \scriptstyle*$}{\kern-5.5pt\lower3pt\hbox{$
     \scriptstyle D$}}}\limits^{\kern0pt\raise.1ex
     \vbox to 0pt{\hbox{$\scriptscriptstyle\leftarrow$}\vss}}}}

\def\vt{\lowfwd t{1.5}1}
\def\tv{\lowbkwd t{1.5}1}

\def\vU{{\vec U}} 
\def\vO{\mathcal{O}}

\def\vm{\lowfwd m{1.5}1}
\def\mv{\lowbkwd m{1.5}1}

\def\vt{\lowfwd t{1.5}1}
\def\tv{\lowbkwd t{1.5}1}

\def\vx{\lowfwd x{1.5}1}
\def\xv{\lowbkwd x{1.5}1}
\def\vy{\lowfwd y{1.5}1}

\def\vE{\lowfwd E{1.5}1}

\def\B{\mathcal{B}}

\def\join{\vee}
\def\meet{\wedge}

\title{Separations of sets}
\author{Nathan Bowler\and Jakob Kneip} %

\begin{document}
\abovedisplayshortskip=-3pt plus3pt
\belowdisplayshortskip=6pt

\maketitle

\begin{abstract}\noindent
	Abstract separation systems are a new unifying framework in which separations of graph, matroids and other combinatorial structures can be expressed and studied.
	
	We characterize the abstract separation systems that have representations as separation systems of graphs, sets, or set bipartitions.
\end{abstract}

\section{Introduction}\label{sec:intro}

Separations of graphs have been studied ever since Robertson and Seymour introduced the notion of {\em tangles} of graphs in~\cite{GMX}. Tangles allow one to describe highly cohesive regions and objects in a graph not directly, say by listing the vertices belonging to that region or object, but rather indirectly by simply stating for each of the graph's low order separations which of the two sides of that separation (most of) the region lies on. The upside of this approach is that the highly cohesive region can be described in this way even if it is a little fuzzy -- e.g., when every individual vertex or edge lies on the wrong side of some low-order separation, such as in the case of a large grid: for every low-order separation, almost all of the grid will lie on one side of it, giving rise to a `consistent' orientation of the low-order separations, but every individual vertex of the grid lies on the `other' side of the separator consisting of just its four neighbours.

In~\cite{AbstractSepSys}, abstract separation systems were introduced to axiomatize, and thereby generalize, this notion of separations and tangles in graphs, so as to make it applicable to cohesive structures also in matroids and other combinatorial objects. This general framework is flexible enough to deal with a plethora of different applications (see~\cite{ProfilesNew,TangleTreeAbstract,MonaLisa} for more). Furthermore, some central structure theorems about tangles can be proved in this setting and then applied to more specific applications, for instance to obtain an elementary proof of Robertson and Seymour's duality theorem connecting tree-width and brambles, which started the study of tangles in~\cite{GMX}.


While treating separations at this abstract level can make the behaviour of their concrete instances more transparent, we need those concrete types of separation to guide our intuition also when we study abstract separation systems. We are therefore led to consider the representation problem familiar from other algebraic contexts: {\em Which abstract separation systems can be represented as separations of graphs?\/} Or as separations of sets such as set bipartitions?

This paper seeks to answer these questions by giving combinatorial characterizations of separation systems of graphs and sets, as well as characterizing those separation systems that come from bipartitions of a set -- an important special case of set separations. Additionally, we give examples of separation systems which are fundamentally different from separation systems of sets or graphs.

The structure of this paper is as follows: in Section~\ref{sec:notation}, we introduce the terms and notation for separation systems used throughout the paper, and make precise what it should mean that a given separation system has the form of set separations. Our main results, characterizing separation systems and universes consisting of separations of a set or bipartitions of a set, are given in Section~\ref{sec:implementations}. Finally, in Section~\ref{sec:graphs}, we treat the case of separations of graphs.

\section{Separation systems}\label{sec:notation}

This paper assumes familiarity with~\cite{AbstractSepSys} and uses the same terms and basic notation. In addition, we shall be using the following terms.

For a separation system $ \vS $ we write $ \Small(\vS) $ for the set of small separations of $ \vS $: those $ \vs\in\vS $ with $ \vs\le\sv $. We call a separation $ \vs $ {\em co-small} if its inverse is small, that is, if $ \vs\ge\sv $.

For a set $ V $ we write $ \S(V) $ for the separation system of all separations of the set $ V $, as defined in~\cite{AbstractSepSys}: the separation system consisting of all (unoriented) separations of the form $ \{A,B\} $, where $ A $ and $ B $ are subsets of $ V $ with $ {A\cup B=V} $, with orientations $ (A,B) $ and $ (B,A) $, where $ (A,B)\le(C,D) $ for oriented separations if and only if $ A\sub C $ and $ D\sub B $. We write $ \U(V) $ for the universe of all separations of the set $ V $: the separation system $ \S(V) $ together with a pairwise supremum~$ \join $ and infimum~$ \meet $ given by
\[ (A,B)\join(C,D)=(A\cup C\,,\,B\cap D) \]
and
\[ (A,B)\meet(C,D)=(A\cap C\,,\,B\cup D) \]
for oriented separations $ (A,B),(C,D)\in\U(V) $.

Furthermore for a set $ V $ we let $ \S\B(V) $ be the separation system of bipartitions of the set $ V $: the sub-system of $ \S(V) $ consisting of all those separations $ (A,B) $ with $ A\cap B=\emptyset $. Similarly, we write $ \U\B(V) $ for the universe of bipartitions of the set $ V $: the sub-universe of $ \U(V) $ containing all separations $ (A,B) $ with $ A $ and $ B $ disjoint. Note that we do not insist that $ A $ and $ B $ be non-empty, so $ \{\emptyset\,,\,V\}\in\U\B(V) $ for all sets $ V $.\\

A map $ f\colon\vS\to\vSdash $ is a {\em homomorphism} of separation systems $ \vS $ and $ \vSdash $ if it commutes with the involutions of $ \vS $ and $ \vSdash $ and is order-preserving. Formally, $ f $ {\em commutes with the involutions} of $ \vS $ and $ \vSdash $ if $ f(\sv)=(f(\vs))^* $ for all $ \vs\in\vS $. The map $ f $ is {\em order-preserving} if $ f(\vr)\le f(\vs) $ whenever $ \vr\le\vs $ for all $ \vr,\vs\in\vS $. An {\em isomorphism} of separation systems $ \vS $ and $ \vSdash $ is a bijective homomorphism $ f\colon\vS\to\vSdash $ whose inverse is also a homomorphism. Two separation systems $ \vS,\vSdash $ are {\em isomorphic} if there is an isomorphism $ f\colon\vS\to\vSdash $.

Similarly, a map $ f\colon\vU\to\vUdash $ is a {\em homomorphism} of universes $ \vU $ and $ \vUdash $ if it commutes with the involutions, joins and meets of $ \vU $ and $ \vUdash $. The map $ f $ {\em commutes with the joins and meets} of $ \vU $ and $ \vUdash $ if $ f(\vr\join\vs)=f(\vs)\join f(\vr) $ and $ f(\vr\meet\vs)=f(\vr)\meet f(\vs) $. An {\em isomorphism} of universes, then, is a bijective homomorphism of universes whose inverse is also a homomorphism. Two universes $ \vU,\vUdash $ are {\em isomorphic} if there is an isomorphism $ f\colon \vU\to\vUdash $. Clearly, every isomorphism of universes is also an isomorphism of separation systems.\\

With the above terms and notation we are now able to formally define what it shall mean that a separation system can be implemented by set separations. Given a separation system $ \vS $, we say that $ \vS $ can be {\em implemented by set separations} if there are a set $ V $ and a sub-system $ \vSdash $ of $ \S(V) $ such that $ \vS $ and~$ \vSdash $ are isomorphic. Similarly, we say that $ \vS $ can be {\em implemented by bipartitions} (of a set) if there are a set $ V $ and a sub-system $ \vSdash $ of $ \S\B(V) $ such that $ \vS $ and~$ \vSdash $ are isomorphic.

If a separation system $ \vS $ can be implemented by set separations or by bipartitions, we call both $ \vSdash $ and the isomorphism $ f\colon\vS\to\vSdash $ witnessing this an {\em implementation} of $ \vS $ by set separations or by bipartitions, respectively.

Finally, for a universe $ \vU $, we say that $ \vU $ can be {\em strongly implemented by set separations} if there are a set $ V $ and a sub-universe $ \vUdash $ of $ \U(V) $ such that~$ \vU $ and~$ \vUdash $ are isomorphic universes. Similarly, we say that~$ \vU $ can be {\em strongly implemented by bipartitions} if there are a set~$ V $ and a sub-universe $ \vUdash $ of $ \U\B(V) $ such that $ \vU $ and~$ \vUdash $ are isomorphic.

If a universe $ \vU $ can be strongly implemented by set separations or by bipartitions, we call both $ \vUdash $ and the isomorphism $ f\colon\vU\to\vUdash $ witnessing this a {\em strong implementation} of $ \vU $ by set separations or by bipartitions, respectively.

Note that, to show that a separation system $ \vS $ can be implemented by set separations or by bipartitions, it suffices to find a ground-set $ V $ and an injective homomorphism~$ f $ from $ \vS $ to $ \S(V) $ or to $ \S\B(V) $ which is an isomorphism between $ \vS $ and its image $ f(\vS) $. In Section~\ref{sec:implementations}, most of the proofs will take this approach.

\section{Set separations and bipartitions of sets}\label{sec:implementations}

In this section we shall characterize those separation systems that can be implemented by separations of sets or by bipartitions of sets. We start with a simple observation regarding the shape of small separations in set separation systems:

\begin{LEM}\label{lem:smalls}
	For any set $ V $, the small separations in $ \S(V) $ and $ \U(V) $ are those of the form $ (A,V) $.
\end{LEM}

\begin{proof}
	Such separations are clearly small, since $ A\sub V $. On the other hand, if $ (A,B) $ is small then we have $ (A,B)\le(B,A) $ and so $ A\sub B $. But this implies~$ {B=A\cup B=V }$.
\end{proof}

By Lemma~\ref{lem:smalls} the small separations in a separation system of sets with ground-set $ V $ have the following property: for every pair $ (A,V),(B,V)\in\U(V) $ we have $ (A,V)\le(B,V)^*=(V,B) $. We will show below that this property characterizes separation systems of sets, so let us make it formal:

A separation system $ \vS $ is {\em scrupulous} if for every pair $ \vr,\vs\in\vS $ of small separations we have $ \vr\le\sv $.

Using the above observation we can characterize the set separation systems as follows:

\begin{THM}\label{thm:char-sets}
	A separation system $ \vS $ can be implemented by set separations if and only if it is scrupulous.
\end{THM}

\begin{proof}
	First we check that $\S(V)$ is scrupulous for any set $V$, from which it follows that any subsystem is scrupulous and so that any $\vS$ which can be implemented by set separations is scrupulous. Let $(A, V)$ and $(A', V)$ be small separations of $\S(V)$. Then $A \subseteq V$ and $A' \subseteq V$, so $(A, V) \leq (V, A')$.
	
	Now suppose that $\vS$ is scrupulous. Let $V$ be the set of all non-co-small elements of $ \vS $. For any $\vs \in \vS$ let
	\[A_\vs := \{\vx \in V \mid \vx \not \geq \vs\}\]
	and
	\[i(\vs) := (A_\vs, A_\sv).\]
	There can't be any $\vx \in V \sm (A_\vs \cup A_\sv)$, since then we would have both $\vx \geq \vs$ and $\vx \geq \sv$, so that $\xv \leq \vs \leq \vx$, which contradicts the fact that~$\vx\in V$ isn't co-small. Thus $i(\vs) \in \S(V)$ for any $\vs \in \vS$. We shall show that $i$ is an implementation of $\vS$ by set separations. It is clear from the definition that $i$ is a homomorphism of separation systems, so it remains to check that it is an isomorphism onto its image. That is, we must show that $i(\vs) \leq i(\vt)$ implies that $\vs \leq \vt$.
	
	So suppose that $i(\vs) \leq i(\vt)$, that is, $ A_\vs\sub A_\vt $ and $ A_\tv\sub A_\sv $. As $\vt\not \in A_\vt$ we have $\vt \not \in A_\vs$. If $\tv$ is not small then $\vt \in V$ and it follows that $\vs \leq \vt$. Similarly, $\sv \not \in A_\sv$ and thus $\sv \not \in A_\tv$, so if $\vs$ is not small then $\sv \in V$ and hence $\vs \leq \vt$. But we also have $\vs \leq \vt$ in the remaining case that  $\vs$ and~$\tv$ are both small, because $\vS$ is scrupulous.
\end{proof}

Not every separation system is scrupulous, as the next example shows:

\begin{EX}\label{ex:separations-sets}
	Let $ \vS $ be the separation system consisting of the separations $ \{\vr,\rv\} $ and $ \{\vs,\sv\} $, with the relations $ \vr\le\rv $ as well as $ \vs\le\sv $ and no further (non-reflexive) relations. Then $ \vr $ and $ \vs $ are small separations with $ \vr\not\le\sv $, so~$ \vS $ is not scrupulous and hence cannot be implemented by set separations.
\end{EX}

Example~\ref{ex:separations-sets} demonstrates how any separation system can be modified so as to not have an implementation by set separations: given a scrupulous separation system $ \vSdash $, one can make this system non-scrupulous by adding a copy of $ \vS $ from Example~\ref{ex:separations-sets} to $ \vSdash $, where separations from $ \vSdash $ are incomparable to those from the copy of $ \vS $. The resulting larger separation system will be non-scrupulous and hence have no implementation by set separations.

However, modifying universes of separations to make them non-scrupulous is not as straightforward as for separation systems due to the existence of joins and meets of any two separations. For universes, being scrupulous is equivalent to another condition on the structure of the small separations:

\begin{LEM}\label{lem:scrup}
	Let $ \vU $ be a universe. Then the following are equivalent:
	\begin{enumerate}
		\item[\tn{(i)}] $ \vU $ is scrupulous, i.e. $ \vs\le\tv $ for all small $ \vs,\vt\in\vU $;
		\item[\tn{(ii)}] $ (\vs\join\vt) $ is small for all small $ \vs,\vt\in\vU $;
		\item[\tn{(iii)}] $ (\vs\meet\vt) $ is co-small for all co-small $ \vs,\vt\in\vU $.
	\end{enumerate}
\end{LEM}

\begin{proof}
	To see that (i) implies (ii), let $ \vs,\vt\in\vU $ be two small separations with $ \vs\le\tv $ and $ \vt\le\sv $. As $ \vs $ is small we have $ \vs\le\sv $, so $ \vs\le(\sv\meet\tv) $. Similarly we have $ \vt\le\tv $ by assumption and hence $ \vt\le(\sv\meet\tv) $. But this implies $ (\vs\join\vt)\le(\sv\meet\tv)=(\vs\join\vt)^* $ and hence (ii).
	
	To see that, conversely, (ii) implies (i), let $ \vs,\vt\in\vU $ be two small separations for which $ (\vs\join\vt) $ is small. Then
	\[ \vs\le(\vs\join\vt)\le(\vs\join\vt)^*=(\sv\meet\tv)\le\tv. \]
	
	Finally, for the equivalence of (ii) and (iii), note that for all $ \vs,\vt\in\vU $ we have $ (\vs\join\vt)^*=(\sv\meet\tv) $ by De Morgan's law, which immediately implies the desired equivalence.
\end{proof}

Typically, the second condition in Lemma~\ref{lem:scrup} is slightly easier to work with than the first, and we shall use it in our proof of Theorem~\ref{thm:char-sets}'s analogue for universes.

To prove a characterization of universes which can be strongly implemented by set separations we shall need the following technical lemma, which is more about lattices than about separation systems:

\begin{LEM}\label{lem:dist}
	Let $ L $ be a distributive lattice and let $ x,s $ and $ t $ be elements of $ L $ with $ s\meet x\le t\meet x $ and $ s\join x\le t\join x $. Then $ s\le t $.
\end{LEM}

\begin{proof}
	By elementary calculations we have
	\begin{align*}
	s&=s\join(s\meet x)\\
	&\le s\join(t\meet x)\\
	&=(s\join t)\meet(s\join x)\\
	&\le(t\join s)\meet(t\join x)\\
	&=t\join(s\meet x)\\
	&\le t\join(t\meet x)\\
	&=t\,.
	\end{align*}
\end{proof}

We are now ready to prove an analogue of Theorem~\ref{thm:char-sets} for universes. As every strong implementation of a universe $ \vU $ by set separations is also an implementation of $ \vU $, viewed as a separation system without joins and meets, every universe which can be strongly implemented using separations of sets must be scrupulous by Theorem~\ref{thm:char-sets}. However, being scrupulous is not a sufficient condition for a universe to have a strong implementation by set separations: for every set $ V $, the universe $ \U(V) $ as well as all sub-universes of it are (easily seen to be) distributive since intersections and unions of sets are distributive.

So, given a distributive and scrupulous universe $ \vU $, how can we find a strong implementation of $ \vU $? Let us first suppose that $ \vU $ already is a sub-universe of~$ \U(V) $ for some set $ V $, and see whether we can describe $ V $ and each $ (A,B)\in\vU $ just in terms of $ \vU $ itself, without making use of $ V $.

To this end, for each $ v\in V $ let $ A_v $ be the set of all $ (A,B)\in\vU $ with $ v\in A $, and let $ V' $ be the set of all those $ A_v $. Then we can write any separation $ (A,B)\in\vU $ as
\[(A,B)=\braces{\{v\in V\mid (A,B)\in A_v\},\{v\in V\mid (B,A)\in A_v\}}.\]
Thus we can define a map $ f\colon\vU\to\U(V') $ as
\[(A,B)\mapsto(\{A_v\in V'\mid(A,B)\in A_v\}\,,\,\{A_v\in V'\mid(B,A)\in A_v\}),\]
and it is easy to check that the map $ f $ is an isomorphism of universes between~$ \vU $ and its image in $ \U(V') $. Thus, the ground-set $ V' $ we defined can be used to obtain a strong implementation of $ \vU $ by separations of sets.

In order to mimic this approach in the general case where we do not know already that $ \vU $ is a sub-universe of some $ \U(V) $, we need to find a collection $ V' $ of sets $ X\sub\vU $ where the sets $ X\in V' $ behave similarly to the sets $ A_v $ above. To do this, we shall find some combinatorial properties of the sets $ A_v $, and then take $ V' $ as the set of all $ X\sub\vU $ which have those combinatorial properties.

In the scenario above where $ \vU $ is a sub-universe of some $ \U(V) $, the first notable property of a set $ A_v=\{(A,B)\in\vU\mid v\in A\} $ for some $ v\in V $ is that $ A_v $ is up-closed: if $ (A,B)\in A_v $ and $ (C,D)\ge(A,B) $, then $ v\in A $ and $ A\sub C $, hence $ v\in C $ and $ (C,D)\in A_v $. Furthermore $ A_v $ is closed under taking meets: if $ (A,B),(C,D)\in A_v $, then $ v\in A $ and $ v\in C $, so $ v\in A\cap C $ and hence
\[(A,B)\meet(C,D)=(A\cap C\,,\,B\cup D)\in A_v.\]
Similarly, we get that the complement of $ A_v $ in $ \vU $ is down-closed and closed under taking joins. Finally, we can say something about the relationship between $ A_v $ and the small separations of $ \vU $: namely, that $ A_v $ contains the inverse of every small separation of $ \vU $. This is because the small separations of $ \vU $ have the form $ (A,V) $, so $ (V,A)\in A_v $ for all of them.

By taking $ V' $ as the set of all $ X\sub\vU $ which have the five properties from the last paragraph, we can prove Theorem~\ref{thm:char-sets-universe}:

\begin{THM}\label{thm:char-sets-universe}
	A universe of separations $ \vU$ can be strongly implemented by set separations if and only if it is distributive and scrupulous.
\end{THM}

\begin{proof}
	If $\vU$ has a strong implementation then it is scrupulous by Lemma~\ref{thm:char-sets} and distributive because $\U(V)$ is distributive for every $V$.
	
	Now suppose that $ \vU $ is distributive and scrupulous. Let $ V $ be the set of all~$ X\sub \vU $ such that
	\begin{enumerate}[$ (1) $]
		\item $ X $ is up-closed in $ \vU $ and closed under taking meets;
		\item $ \vU\sm X $ is down-closed in $ \vU $ and closed under taking joins;
		\item $ X $ contains all co-small elements of $ \vU $.
	\end{enumerate}
	For any $ \vs\in\vU $ we have $ (\vs\join\sv)^*=\sv\meet\vs\le\vs\join\sv $, so $ \vs\join\sv $ is co-small for all $ \vs\in\vU $. Given $ X\in V $ and $ \vs\in\vU $ we thus have $ (\vs\join\sv)\in X $. Therefore $ X $ must contain at least one of $ \vs $ and $ \sv $, as we cannot have $ \vs,\sv\in\vU\sm X $ by~$ (2) $.
	
	For any $ \vs\in\vU $ let $ A_\vs:=\{X\in V\mid \vs\in X\} $ and $ f(\vs):=(A_\vs,A_\sv) $. By the above argument we have $ A_\vs\cup A_\sv=V $, so $ f $ takes its image in $ \U(V) $. This map clearly commutes with the involution, and by~$ (1) $ and $ (2) $ it also commutes with $ \meet $ and~$ \join $. It remains to show that $ f$ is injective. For this let $ \vs,\vt\in\vU $ with $ \vs\ne\vt $ be given; we shall show that $ f(\vs)\ne f(\vt) $. By switching their roles if necessary we may assume that~$ \vs\not\le\vt $. \\
	\\
	{\bf Claim 1:} If there is no co-small $ \vx_1\in\vU $ such that $ \vs\meet\vx_1\le\vt $ then $ A_\vs\not\sub A_\vt $.\\
	
	To see this, we wish to find a pair $ (X,Y) $ of disjoint subsets of $ \vU $ such that
	\begin{enumerate}[(I)]
		\item $ \vs\in X $ and $ X $ is up-closed in $ \vU $ and closed under taking meets;
		\item $ \vt\in Y $ and $ Y $ is down-closed in $ \vU $ and closed under taking joins;
		\item $ X $ contains all co-small elements of $ \vU $.
	\end{enumerate}
	Call such a pair $ (X,Y) $ {\it good}. We will show later that a maximal good pair $ (X,Y) $ will then have $ X\in V $ and hence witness that $ {A_\vs\not\sub A_\vt} $; let us show first that some good pair exists. To see this, let
	\[A:=\bigcup_{\xv\in\Small(\vU)}\ucl (\vs\meet\vx)\]
	and $ B:=\dcl (\vt) $. Then $ A $ satisfies (I) by Lemma~\ref{lem:scrup} and (III) by construction, and~$ B $ clearly satisfies (II). Thus $ (A,B) $ is a good pair provided $ A $ and $ B $ are disjoint. Suppose they are not disjoint; then $ \vt\in A $ and hence there is some co-small $ \vx_1\in\vU $ with $ \vt\in\ucl (\vs\meet\vx_1) $, which contradicts the premise of Claim~1.
	
	Now let $ (X,Y) $ be an inclusion-wise maximal good pair, which exists by Zorn's Lemma. We wish to show $ Y=\vU\sm X $ as that would imply $ X\in V $ and in particular $ X\in A_\vs\sm A_\vt $. Suppose there exists $ \vr\in\vU\sm (X\cup Y) $. By the maximality of $ X $ there are $ \vx_1\in X $ and $ \vy_1\in Y $ with $ \vr\meet\vx_1\le\vy_1 $, and by the maximality of $ Y $ there are $ \vx_2\in X $ and $ \vy_2\in Y $ with $ \vx_2\le\vr\join\vy_2 $. Set $ \vx:=\vx_1\meet\vx_2 $ and $ \vy:=\vy_1\join\vy_2 $. Then $ \vx\in X $ and $ \vy\in Y $ with $ \vx\meet\vr\le\vy\meet\vr $ and $ \vx\join\vr\le\vy\join\vr $. Lemma~\ref{lem:dist} now implies $ \vx\le\vy $, but this contradicts the fact that $ X $ is up-closed and disjoint from $ Y $. Therefore $ Y=\vU\sm X $ and $ X\in V $, which proves the claim.\\
	\\
	{\bf Claim 2:} If there is no co-small $ \vx_2\in\vU $ such that $ \tv\meet\vx_2\le\sv $ then $ A_\tv\not\sub A_\sv $.\\
	\\
	As $ \vs\not\le\vt $ is equivalent to $ \tv\not\le\sv $ Claim~2 follows in exactly the same way as Claim~1.\\
	\\
	{\bf Claim 3:} There cannot be co-small $ \vx_1,\vx_2\in\vU $ such that
	\begin{align*}
	\vs\meet\vx_1&\le\vt,\\
	\tv\meet\vx_2&\le\sv.
	\end{align*}
	To see this, suppose $ \vx_1,\vx_2 $ are as in the claim. Let $ \vx:=\vx_1\meet\vx_2 $; this is a co-small separation by Lemma~\ref{lem:scrup}. We then have $ \vs\meet\vx\le\vt $ and $ \tv\meet\vx\le\sv $. Applying the involution to the latter inequality gives $ \vs\le\vt\join\xv\le\vt\join\vx $. Therefore we get that $ \vs\meet\vx\le\vt\meet\vx $ as well as $ \vs\join\vx\le\vt\join\vx $, which by Lemma~\ref{lem:dist} contradicts the assumption that $ \vs\not\le\vt $. This proves the claim.\\
	\\
	From Claim~3 it follows that the assumption of at least one of Claim~1 or Claim~2 must be satisfied, and hence $ f(\vs)\ne f(\vt) $, which completes the proof.
\end{proof}

The next example shows that the assumption of distributivity in Theorem~\ref{thm:char-sets-universe} is indeed necessary, as there are abstract universes of separations which are not distributive:

\begin{EX}\label{ex:universe-sets}
	Let $ L $ be an arbitrary non-distributive lattice. Let $ \vU $ be the separation system consisting of an unoriented separation $ \{\vr,\rv\} $ for each $ {r\in L} $, with the following relations: $ \vr\le\vs $ and $ \sv\le\rv $ if and only if $ r\le s $ in $ L $, and additionally $ \vr\le\sv $ for all $ r,s\in L $. As $ L $ is a lattice every pair of separations in~$ \vU $ has a meet and a join, so $ \vU $ is a universe. As $ L $ is non-distributive,~$ \vU $ is non-distributive by construction, too. Therefore there can be no strong implementation of $ \vU $ by set separations. Furthermore $ \vU $ is scrupulous, showing that the distributivity cannot be omitted from Theorem~\ref{thm:char-bips-universe}.
\end{EX}

Let us now turn to the topic of (strong) implementations by bipartitions. From Lemma~\ref{lem:smalls} it follows that the only small separation of $ \S\B(V) $ is $ (\emptyset,V) $. This separation is not only small, it is also the least element of $ \S\B(V) $. So let us call a separation system $ \vS $ {\em fastidious} if we have $ \vs\le\vt $ for all small $ \vs\in\vS $ and all $ \vt\in\vS $. Clearly every fastidious separation system has at most one small separation, and every separation system with an implementation by bipartitions of sets must be fastidious. Furthermore, every fastidious separation system is scrupulous.

Somewhat surprisingly, Theorem~\ref{thm:char-sets-universe} directly implies that every distributive and fastidious universe has a strong implementation by bipartitions of sets:

\begin{THM}\label{thm:char-bips-universe}
	A universe of separations $ \vU$ can be strongly implemented by bipartitions of sets if and only if it is distributive and fastidious.
\end{THM}

\begin{proof}
	Since the only small bipartition of a set $ V $ is $ (\emptyset,V) $, any separation system implemented by bipartitions is fastidious. Now suppose that $ \vU $ is a distributive and fastidious universe. Then $ \vU $ is scrupulous, so by Theorem~\ref{thm:char-sets-universe} there is a strong implementation $ f\colon\vU\to\vUdash $ of $ \vU $ by set separations, where $ \vUdash $ is a sub-universe of $ \U(V) $ for some set $ V $. As $ (\vs\join\sv) $ is co-small for any $ \vs\in\vU $ there exists a co-small separation $ \vr\in\vU $. Since~$ \vU $ is fastidious this $ \vr $ must be the greatest element of $ \vU $. In particular $ \vr $ is the unique co-small element of $ \vU $ and we have $ \vs\join\sv=\vr $ for all $ \vs\in\vU $. By Lemma~\ref{lem:smalls}, the image of $ \vr $ under $ f $ is of the form $ f(\vr)=(V,X) $ for some $ X\sub V $. Given any $ \vs\in\vU $ and $ (A,B):=f(\vs) $ we thus have $ (A,B)\join(B,A)=(V,X) $, and in particular $ A\cap B=X $.
	
	Consider the map $ g\colon \vUdash\to\U(V\sm X) $ defined by
	\[g(A,B):=(A\sm X\,,\,B\sm X).\]
	As $ X=A\cap B $ for all $ (A,B)\in f(\vU) $ the map $ g $ is a strong implementation of~$ \vUdash $ by bipartitions of $ V\sm X $. The map $ h\colon\vU\to\U(V\sm X) $ defined by $ h=g\circ f $ is thus a strong implementation of $ \vU $ by bipartitions of $ (V\sm X) $.
\end{proof}

The next example shows that the assumption of distributivity in Theorem~\ref{thm:char-bips-universe} cannot be omitted, as there are fastidious universes which are not distributive:

\begin{EX}\label{ex:universe-bipartitions}
	Let $ \vU $ be the universe consisting of three unoriented separations $ \{\vr,\rv\},\{\vs,\sv\} $ and $ \{\vt,\tv\} $ with the relations
	\[\vr\le\vs\le\vt\le\rv\qquad\qquad\tn{and}\qquad\qquad\vr\le\tv\le\sv\le\rv.\]
	Then $ \vU $ is a fastidious universe with $ \vr $ as the least element, but we have
	\[\vt=(\vs\join\sv)\meet\vt\ne(\vs\meet\vt)\join(\sv\meet\vt)=\vs,\]
	so $ \vU $ is not distributive, showing that the assumption of distributivity in Theorem~\ref{thm:char-bips-universe} is necessary.
\end{EX}

Example~\ref{ex:universe-bipartitions} is a modification of the pentagon lattice $ N_5 $, which is an elementary example of a non-distributive lattice. The other prototypical non-distributive lattice, the diamond lattice $ M_3 $, can be turned into an example of a fastidious non-distributive universe in a similar fashion.%
\COMMENT{
	The way to do that is to have three unoriented separations $ \{\vr,\rv\},\{\vs,\sv\} $ and $ \{\vt,\tv\} $ where $ \vr $ is the least and $ \rv $ the greatest element, with no further relations.
}

The power of the theory of universes of bipartitions can be seen in~\cite{ProfilesNew}, where it is applied to obtain an neat proof of the existence of Gomory-Hu trees in finite graphs.

Interestingly, the conclusion of Theorem~\ref{thm:char-sets} that every scrupulous separation system can be implemented by sets does not directly imply that every fastidious separation system has an implementation by bipartitions, even though the analogous implication is true for universes as seen in the proof of Theorem~\ref{thm:char-bips-universe}. The next example illustrates this:

\begin{EX}\label{ex:seps-no-implication}
	Let $ V $ be the three-element set $ \{x,y,z\} $ and $ \vS $ the separation system containing the three unoriented set separations
	\[\menge{\menge{x,y},\menge{x,z}},\qquad\menge{\menge{x,y},\menge{y,z}}\qquad\tn{and}\qquad\menge{\menge{x,z},\menge{y,z}}.\]
	Then $ \vS $ has no small separations and hence is a fastidious separation system. However as every $ v\in V $ lies in $ A\cap B $ for some $ (A,B)\in\vS $ it is not possible to obtain an implementation of $ \vS $ by bipartitions of sets by deleting those elements from $ V $, as we did in the proof of Theorem~\ref{thm:char-bips-universe}.
\end{EX}
\COMMENT{
	J: The above example is a three-star, which has an obvious implementation by bipartitions of sets. This implementation ($ (x,yz),(xy,z) $ and $ (z,xy) $) can even be found by manipulating the separations of $ \vS $, but as far as I can see there is no underlying reason why something like this should always be possible.
}

In view of Example~\ref{ex:seps-no-implication}, if we want to prove that every fastidious separation system can be implemented by bipartitions of sets, we cannot use Theorem~\ref{thm:char-sets} but need to find a direct proof. The following example from~\cite{TreeSets} points us in the right direction:

\begin{EX}\label{ex:orientations}
	Let $ T $ be a tree and $ \vS $ the edge tree set of $ T=(V,E) $ as defined in~\cite{TreeSets}: the separation system on the set $ \vE(T) $ of oriented edges of $ T $ in which $ (v,w)<(x,y) $ for oriented edges $ (v,w),(x,y)\in\vE(T) $ if and only if $ \menge{v,w}\ne\menge{x,y} $ and the unique $ \menge{v,w} $--$ \menge{x,y} $-path in $ T $ joins $ w $ to $ x $. Then $ \vS $ contains no small separations and is therefore fastidious.
	
	The separation system $ \vS $ has a natural implementation by bipartitions of sets: for each separation $ \vs=(v,w)\in\vS $, the removal of the edge $ \{v,w\} $ from $ T $ partitions the vertices of $ T $ into exactly two connected components $ C_{(v,w)} $ and $ C_{(w,v)} $,	which contain $ v $ and $ w $ respectively. Furthermore, we have $ (v,w)\le(x,y) $ for separations in $ \vS $ if and only if $ C_{(v,w)}\sub C_{(x,y)} $. Thus it is easy to check that the map $ f\colon\vS\to\S\B(V) $ defined by
	\[f\braces{(v,w)}:=(C_{(v,w)}\,,\,C_{(w,v)})\]
	is an isomorphism between $ \vS $ and its image in $ \S\B(V) $ and hence an implementation of $ \vS $ by bipartitions.
	
	Additionally, the vertex set $ V $ of $ T $ can be described wholly in terms of $ \vS $, without referencing the tree $ T $: every vertex $ v $ of $ T $ induces a unique consistent orientation $ O_v $ of $ \vS $ by orienting each edge in $ E(T) $ towards $ v $.\footnote{Additionally, if $ T $ is finite, then every consistent orientation of $ \vS $ (viewed as orientation of $ E $) points to a unique vertex: each such orientation has precisely one sink.} It is easy to see, then, that for $ v\in V $ and $ (x,y)\in\vS $ we have $ v\in C_{(x,y)} $ if and only if $ (y,x)\in O_v $.
	
	This observation leads to another way of implementing $ \vS $ by bipartitions of a set. Let $ \vO=\vO(\vS) $ be the set of all consistent orientations of $ \vS $, and for $ \vs\in\vS $ let $ \vO_\vs $ be the set of all $ O\in\vO $ which contain $ \vs $. Let us define the map $ g\colon\vS\to\B\S(\vO) $ by setting
	\[g(\vs):=\braces{\vO_\sv,\vO_\vs}.\]
	Using the Extension Lemma~(\cite[Lemma 4.1]{AbstractSepSys}) it is straightforward to check that~$ g $ is an isomorphism between $ \vS $ and its image in $ \S\B(\vO) $ and hence an implementation of $ \vS $ by bipartitions of a set.
\end{EX}

For an arbitrary separation system $ \vS $ let $ \vO=\vO(\vS) $ be the set of consistent orientations of $ \vS $. For $ \vs\in\vS $ let us write $ \vO_\vs $ for the set of all orientations $ O\in\vO $ which contain $ \vs $. For any non-degenerate $ \vs\in\vS $ we have $ \vO=\vO_\vs\,\dot{\cup}\,\vO_\sv $ as orientations are anti-symmetric.

Remarkably, the map $ g $ defined in Example~\ref{ex:orientations} above which maps a separation $ \vs $ to $ \braces{\vO_\sv\,,\,\vO_\vs} $ is an isomorphism onto its image in $ \S\B(\vO) $ for all regular separation systems $ \vS $, not only those which are the edge tree set of some tree as in Example~\ref{ex:orientations}:

\begin{THM}[\cite{TreeSets}, Theorem~5.2]\label{thm:regular-bips}
	Given any regular separation system $ \vS $, the map $ f\colon\vS\to\S\B(\vO) $ given by $ \vs\mapsto(\vO_\sv,\vO_\vs) $ is an isomorphism of separation systems between $ \vS $ and its image in $ \S\B(\vO) $.
\end{THM}

Due to the context of~\cite{TreeSets}, in~\cite{TreeSets} Theorem~\ref{thm:regular-bips} is only formulated for separation systems $ \vS $ that are nested (i.e. in which every two unoriented separations can be oriented so as to be comparable). However, this assumption is not necessary, and indeed, the proof of Theorem~\ref{thm:regular-bips} given in~\cite{TreeSets} does not use it.

Since a fastidious separation system has at most one small separation and is thus `almost regular' we can utilize Theorem~\ref{thm:regular-bips} to show that every fastidious separation system has an implementation by bipartitions of sets:

\begin{THM}\label{thm:char-bips}
	A separation system $ \vS $ can be implemented by bipartitions of sets if and only if it is fastidious.
\end{THM}

\begin{proof}
	Since the only small bipartition of a set $ V $ is $ (\emptyset,V) $, any separation system implemented by bipartitions is fastidious.
	
	Now suppose that $ \vS $ is a fastidious separation system. If $ \vS $ is regular the assertion follows immediately from Theorem~\ref{thm:regular-bips}, so let us suppose that $ \vS $ is not regular. As $ \vS $ is fastidious it then has a unique small separation $ \vs $.
	
	We consider first the case that $ s $ is the only separation in $ \vS $, i.e. that $ \vS=\{\vs,\sv\} $. If $ s $ is degenerate, so $ \vs=\sv $, then $ \vS $ is isomorphic to $ \S\B(\emptyset) $. And if $ \vs\ne\sv $ then for any non-empty set $ V $ the map $ f\colon\vS\to\S\B(V) $ which maps $ \vs $ to $ (\emptyset,V) $ and $ \sv $ to $ (V,\emptyset) $ is an implementation of $ \vS $ by bipartitions of $ V $.
	
	Let us now suppose that $ \vS $ has elements other than $ s $. Then $ s $ cannot be degenerate: for if $ \vs=\sv $ then, for each $ \vt\in\vS $, we would have both $ \vs\le\vt $ and $ \sv=\vs\le\tv $ by assumption. The latter inequality is equivalent to $ \vt\le\vs $ and hence implies $ \vs=\vt $, contrary to the assumption that $ s $ is not the only separation in $ \vS $.
	
	 As $ \vs $ is the only small separation of $ \vS $ the sub-system $ \vSdash$ obtained from $ \vS $ by deleting~$ s $ is a regular separation system. By applying Theorem~\ref{thm:regular-bips} to $ \vSdash $ we obtain an implementation $ f' $ of $ \vSdash $ using $ V=\vO(\vSdash) $ as a ground-set. Let $ f\colon\vS\to\U\B(V) $ be the extension of $ f' $ to~$ \vS $ which maps $ \vs $ to $ (\emptyset,V) $ and $ \sv $ to $ (V,\emptyset) $. As $ \vs\le\vt $ for each $ \vt\in\vS $ and $ \vs $ is non-degenerate $ f $ clearly is an implementation of $ \vS $ by bipartitions of $ V $.
\end{proof}

\section{Graph separations}\label{sec:graphs}

Tangles, the study of which motivated the introduction of abstract separation systems in~\cite{AbstractSepSys}, were introduced in~\cite{GMX} in terms of graph separations. Therefore this instance of separation systems is of special interest. In this final section we shall characterize those separation systems that can be implemented by graph separations.

Given an undirected graph $ G=(V,E) $, we denote with $ \U(G) $ the universe of graph separations of $ G $: the sub-universe of $ \U(V) $ consisting of all those separations $ (A,B) $ for which $ G $ contains no edge from $ A\sm B $ to $ B\sm A $.

We say that a universe $ \vU $ has a {\em graphic implementation} if there exists a graph $ G $ such that $ \U(G) $ and $ \vU $ are isomorphic. Note that this notion differs slightly from the notions of implementability in Section~\ref{sec:implementations}: here we ask that $ \vU $ is isomorphic to $ \U(G) $ itself, and not just to some sub-universe of $ \U(G) $. The reason for this difference is that $ \U(G)=\U(V) $ for any graph $ G=(V,E) $ with no edges, and hence asking for universes which are isomorphic to a sub-universe of $ \U(G) $ for some graph $ G $ would be the same as asking for those which are isomorphic to a sub-universe of $ \U(V) $ for some set $ V $. The latter is our notion of strong implementations that we already discussed in Section~\ref{sec:implementations} and Theorem~\ref{thm:char-sets-universe}.

Furthermore in this section we shall only deal with finite universes and separation systems.

The aim of this section is then to characterize those universes of separations which have a graphic implementation.

Let us start with a couple of simple observations. Any graphic implementation is also a strong implementation, hence every $ \vU $ with a graphic implementation must be distributive and scrupulous by Theorem~\ref{thm:char-sets-universe}. Additionally $ \U(G) $ contains each separation of the form $ (X,V) $ for $ X\sub V $, that is to say, $ \U(G) $ contains all small separations of $ \U(V) $. These have a very particular structure: the set of small separations of $ \U(V) $ forms a boolean algebra. In fact we can say a bit more about this algebra: its maximal element will be $ (V,V) $, the only degenerate separation in $ \U(V) $. Therefore if $ \vU $ is to have a graphic implementation then the set of small separations in $ \vU $ must form a boolean algebra whose maximal element is degenerate in $ \vU $. It is easy to check that this latter condition already implies that $ \vU $ is scrupulous.

So, let us show that every finite universe $ \vU $ has a graphic implementation provided $ \vU $ is distributive and its set of small separations forms a boolean algebra with degenerate maximal element.

Our strategy for finding a graph $ G $ whose universe of separations is isomorphic to such a universe $ \vU $ will be as follows: first we shall apply Theorem~\ref{thm:char-sets-universe} to find some a implementation $ f\colon\vU\to\U(V) $ of $ \vU $ for some ground-set $ V $. We will take $ V $ as the vertex-set of our graph $ G $ and need to define the edges of $ G $ in such a way that the image of $ \vU $ under $ f $ in $ \U(V) $ is exactly $ \U(G) $. That means that for every $ \vs\in\vU $ with $ f(\vs)=(A,B) $ we cannot join a vertex from $ A\sm B $ to a vertex in $ B\sm A $, as then $ f(\vs) $ would not be a separation of $ G $. Conversely, if for two vertices $ v,w\in V $ there is no $ \vs\in\vU $ whose image $ (A,B) $ under $ f $ has $ v $ and $ w $ in $ A\sm B $ and $ B\sm A $ respectively then we must make $ v $ and $ w $ adjacent in $ G $ as otherwise there would be a separation of $ G $ which lies outside the image of $ \vU $ under $ f $. Thus we will define $ E(G) $ as the set of all pairs $ \{v,w\} $ of vertices for which there is no $ (A,B) $ in the image $ f(\vU) $ with $ v\in A\sm B $ and $ w\in B\sm A $. This way $ f $ will take its image in $ \U(G) $, but we still need to prove that $ f $ is onto on $ \U(G) $.

For the proof of surjectivity we will make use of the the structure of the set of small separations in $ \vU $. First we shall show that we can choose the strong implementation $ f\colon\vU\to\U(V) $ of $ \vU $ in such a way that every small separation of $ \U(V) $ lies in the image of $ \vU $ under $ f $. Then we shall use this to show that for every non-edge $ vw $ of $ G $ the separation $ \braces{V\sm\{v\}\,,\,V\sm\{w\}} $ lies in the image of $ \vU $ under~$ f $. The surjectivity of $ f $ onto $ \U(G) $ will then follow from the two facts that $ f(\vU) $ is closed under taking joins and meets and that every separation $ (A,B)\in\U(G) $ can be obtained from the separations of the above form through taking joins and meets.\\
\\
Given a finite universe $ \vU $ a separation $ \vs\in\vU $ is {\em atomic} if $ \vs $ is not the least element of $ \vU $, but the only element less than $ \vs $ is the least element of $ \vU $. Recall that we write $ \Small(\vU) $ for the set of small separations of $ \vU $.

The next lemma shows that for finite universes we can choose the strong implementation in Theorem~\ref{thm:char-sets-universe} in such a way that the image of each atomic element of $ \vU $ is of the form $ (\{v\}\,,\,V) $:

\begin{LEM}\label{lem:atomic-implementation}
	Let $ \vU $ be a finite universe. If $ \vU $ is distributive and scrupulous then $ \vU $ has a strong implementation by sets in which the image of every small atomic element of $ \vU $ is of the form $ (\{v\},V) $ for some $ v $ in the ground-set $ V $.
\end{LEM}

\begin{proof}
	Let $ V $ be defined as in the proof of Theorem~\ref{thm:char-sets-universe}: as the set of all $ X\sub\vU $ with the properties that $ X $ is up-closed and closed under taking meets, that $ \vU\sm X $ is down-closed and closed under taking joins, and that $ X $ contains all co-small separations. For $ \vr\in\vU $ let $ A_\vr $ be the set of all $ X\in V $ with~$ \vr\in X $. As the proof of Theorem~\ref{thm:char-sets-universe} goes on to show, the map $ \vr\mapsto\braces{A_\vr\,,\,A_\rv} $ is an isomorphism between $ \vU $ and its image in $ \U(V) $. We shall show that a slight modification of this already is an implementation of $ \vU $ with the desired property that the image of each small atomic separation is of the form $ (\{v\},V) $. %
	\COMMENT{
		In fact, we just need to delete one superfluous element from the ground-set which lies on both sides of all separations.
	}
	
	For this, observe first that $ \vU $ (as a set) is an element of the ground-set $ V $ by definition of $ V $. We have $ \vU\in A_\vr $ and $ \vU\in A_\rv $ for all $ \vr\in\vU $ since $ \vr,\rv\in\vU $. Let $ V':=V\sm\{\vU\} $, and for $ \vr\in\vU $ let $ A'_\vr $ be the set of all $ X\in V' $ that contain~$ \vr $ (equivalently: $ A'_\vr:=A_\vr\sm\{\vU\} $). Then the map $ f\colon\vU\to\U(V') $ with $ \vr\mapsto(A'_\vr\,,\,A'_\rv) $ is an isomorphism between $ \vU $ and its image in $ \U(V') $.
	We claim that this map $ f $ is the desired strong implementation of $ \vU $. For this we just need to show that the image under $ f $ of each small atomic element of $ \vU $ is of the form $ (\{v\},V') $ for some $ v\in V' $. So let $ \vs\in\vU $ be a small atomic element of $ \vU $.
	
	Since $ \vs $ is small its image in $ \U(V') $ is of the form $ (Y,V') $ for some $ Y\sub V' $. Then $ Y $ is non-empty: as $ \vs $ is atomic and thus not the least element of $ \vU $, its image under $ f $ cannot be $ (\emptyset,V') $, the least element of $ \U(V') $. To finish the proof we thus need to show that $ Y $ has at most one element. So let $ X $ be an element of $ Y=A_\vs $. Since $ \vU $ is finite and $ X\in V' $ is up-closed and closed under taking meets, $ X $ is the up-closure of its least element, say $ \vr $. As $ \vs\in X $ we have $ \vr\le\vs $. But thus, as $ \vs $ is atomic, either $ \vr=\vs $ or $ \vr $ is the least element of $ \vU $. The latter would imply $ X=\vU\notin V' $, so we must have $ X=\ucl(\vs) $, proving that~$ A_\vs $ contains exactly one element.
\end{proof}

We are now ready to prove that every finite universe that is distributive and whose small elements form a boolean algebra with degenerate maximal element has a graphic implementation:

\begin{THM}\label{thm:char-graphs}
	A finite universe $ \vU $ has a graphic implementation if and only if it is distributive and $ \Small(\vU) $ is a boolean algebra whose maximal element is degenerate in $ \vU $. 
\end{THM}

\begin{proof}
	It is clear that for any graph $ G=(V,E) $ the universe $ \U(G) $ is distributive and $ \Small(\U(G)) $ is isomorphic to $ \mathcal{P}(V) $, which is a boolean algebra. The maximal small element of $ \U(G) $ is $ (V,V) $, which is degenerate.
	
	Now suppose that $ \vU $ is a distributive universe and that $ \Small(\vU) $ is a boolean algebra whose maximal element is degenerate. Then in particular there is a maximal small element $ \vm $ and for any small elements $ \vr,\vs\in\vU $ we have $ \vr\le\vm=\mv\le\sv $, so that $ \vU $ is scrupulous. Thus by Lemma~\ref{lem:atomic-implementation} there are a set $ V $ and a map $ f\colon\vU\to\U(V) $ such that $ f $ is an isomorphism of universes between $ \vU $ and its image in $ \U(V) $, and for which for every small atomic $ \vs\in\vU $ there is a $ v\in V $ such that $ f(\vs)=(\{v\}\,,\,V) $.
	
	Let us now define $ G=(V,E) $ to be the graph with vertex set $ V $ and an edge from $ v $ to $ w $ if and only if there is no separation $ (A,B) $ in $ f(\vU) $ with $ v\in A\sm B $ and $ w\in B\sm A $. Note that $ E $ is well-defined by this as $ (A,B)\in f(\vU) $ if and only if $ (B,A)\in f(\vU) $. Then $ \U(G) $ is a sub-universe of $ \U(V) $. We claim that $ f $ is an isomorphism of universes between $ \vU $ and $ \U(G) $ and thus a graphic implementation of $ \vU $. As $ f $ is an isomorphism between $ \vU $ and its image in $ \U(V) $ it suffices to show that $ f(\vU)=\U(G) $.
	
	To see that $ f(\vU)\sub\U(G) $, let $ (A,B) $ be any separation in $ f(\vU) $. But then $ (A,B) $ is also a separation of $ G $: for all $ v\in A\sm B $ and $ w\in B\sm A $, by definition of~$ E $, the separation $ (A,B) $ itself witnesses that there is no edge between $ v $~and~$ w $.
	
	For the converse inclusion $ \U(G)\sub f(\vU) $, observe first that $ f(\vU) $ is closed under taking inverses, meets, and joins. Let $ \vm $ be the maximal element of $ \Small(\vU) $. Since~$ \vm $ is degenerate by assumption and $ (V,V) $ is the only degenerate separation of $ \U(G) $ we must have $ f(\vm)=(V,V) $, and since $ \Small(\vU) $ is a boolean algebra, $ \vm $ is the join of small atomic elements. All small atomic separations $ \vs $ of $ \vU $ get mapped to separations of the type $ (\{v\}\,,\,V) $ with $ v\in V $. The join of those separations can only be $ f(\vm)=(V,V) $ if every separation of the form $ (\{v\}\,,\,V) $ lies in $ f(\vU) $. Thus we must have $ (\{v\}\,,\,V)\in f(\vU) $ for every $ v\in V $, from which it follows that $ \Small(\U(V))\sub f(\vU) $: for any $ X\sub V $ we have
	\[ (X,V)=\bigvee_{x\in X}(\{x\}\,,\,V), \]
	with the right hand side lying in $ f(\vU) $. As $ f(\vU) $ is closed under taking inverses we also have $ (V,X)\in f(\vU) $ for every $ X\sub V $.
	
	Let us now show that for every non-edge $ vw $ of $ G $ the separation
	\[{(V\sm\{v\}\,,\,V\sm\{w\})}\]
	lies in $ f(\vU) $. To see this, let $ v $ and $ w $ be two non-adjacent vertices of $ G $. The fact $ v $ and $ w $ are not adjacent means that by definition of $ E $ there is some $ (A,B)\in f(\vU) $ with $ v\in A\sm B $ and $ w\in B\sm A $. But then $ (B,A)\in f(\vU) $, too, and
	\[ (V\sm\{v\}\,,\,V\sm\{w\})=\class{(B,A)\join(V\sm\{v\}\,,\,V)}\meet(V\,,\,V\sm\{w\}), \]
	where the right hand side lies in $ f(\vU) $.
	
	To finish our proof that $ \U(G)\sub f(\vU) $, let $ (A,B)\in\U(G) $ be arbitrary. Then we can write $ (A,B) $ as
	\[ (A,B)=\bigwedge_{v \in V \sm A}\bigvee_{w \in V \sm B}(V\sm\{v\}\,,\,V\sm\{w\}), \]
	where the right hand side lies in $ f(\vU) $ as for all $ v\in V\sm A $ and $ w\in V\sm B $ there can be no edge between $ v $ and $ w $ as $ (A,B) $ is a separation of $ G $, giving $ (V\sm\{v\}\,,\,V\sm\{w\})\in f(\vU) $ by the above argument.
\end{proof}

\bibliographystyle{plain}
\bibliography{collective.bib}

\end{document}